\documentclass[a4paper]{amsart}
\usepackage{amsthm,amssymb}
\usepackage[T1]{fontenc}
\usepackage[utf8]{inputenc}
\usepackage{enumitem}
\usepackage{mathrsfs}
\usepackage{dsfont}
\usepackage{lmodern}
\usepackage{latexsym}
\usepackage{scalerel,stackengine}
\usepackage{nicefrac}
\usepackage{natbib}

\usepackage{graphicx}
\usepackage[colorlinks,citecolor=blue,urlcolor=blue]{hyperref}

\usepackage{tikz}

\usepackage{thmtools}

\declaretheorem[style=definition,name=Definition,qed=$\dashv$]{definition}

\declaretheorem[style=remark,name=Example,qed=$\dashv$]{example}

\numberwithin{equation}{section}

\newtheorem{theorem}{Theorem}[section]
\newtheorem{lemma}[theorem]{Lemma}
\newtheorem{proposition}[theorem]{Proposition}
\newtheorem{corollary}[theorem]{Corollary}

\theoremstyle{remark}
\newtheorem{problem}{Problem}

\newcommand\boldd{\boldsymbol{d}} 
\newcommand{\cond}{\con_{\boldd}} 
\newcommand{\notcond}{\notcon_{\boldd}} 

\newcommand\cl[1]{\overline{#1}} 

\newcommand\lld{\ll_{\boldd}} 
\newcommand{\Basis}{\mathscr{B}} 

\newcommand\frP{\mathfrak{P}}

\title{Grzegorczyk and Whitehead points: the story continues}

\author{Rafa\l\ Gruszczy\'nski, Santiago Jockwich Martinez}

\date{}

\address{Rafa\l\ Gruszczy\'nski\\
Department of Logic\\
Nicolaus Copernicus University in Toru\'n\\
Poland}

\email{gruszka@umk.pl}

\urladdr{www.umk.pl/\textasciitilde gruszka}

\address{Santiago Jockwich Martinez\\
Department of Philosophy, Classics, History of Art and Ideas\\
University of Oslo\\
Norway}

\email{s.j.martinez@ifikk.uio.no}

\newcommand\frB{\mathfrak{B}}

\DeclareMathSymbol{\ppartof}{\mathord}{AMSa}{64}
\DeclareMathSymbol{\partof}{\mathrel}{AMSa}{64}
\DeclareSymbolFont{symbolsC}{U}{txsyc}{m}{n}
\DeclareMathSymbol{\poverl}{\mathord}{symbolsC}{7}
\DeclareMathSymbol{\overl}{\mathrel}{symbolsC}{7}
\DeclareMathSymbol{\pext}{\mathord}{symbolsC}{78}
\DeclareMathSymbol{\npartof}{\mathrel}{symbolsC}{97}
\DeclareMathSymbol{\ningr}{\mathrel}{symbolsC}{64}
\DeclareMathSymbol{\nll}{\mathrel}{symbolsC}{51}

\newcommand\ext{\mathrel{\bot}}

\DeclareMathOperator{\Int}{Int}      
\DeclareMathOperator{\Cl}{Cl}        

\newcommand\iffdef{\;\mathrel{\mathord{:}\mathord{\longleftrightarrow}}\;}

\newcommand\iffslim{\longleftrightarrow}
\newcommand\defeq{\coloneqq}

\newcommand\rarrow{\longrightarrow}
\newcommand\larrow{\longleftarrow}

\newcommand\fun{\ensuremath{\rightarrow}} 
\newcommand{\power}{\mathcal{P}}

\newcommand{\zero}{\mathsf{0}}
\newcommand{\one}{\mathsf{1}}

\let\separ=\msep

\newcommand{\con}{\mathrel{\mathsf{C}}} 
\DeclareMathOperator{\cont}{\mathsf{C}_{\mathrm{T}}} 

\newcommand\mathbackslash{\raisebox{.4pt}{\texttt{/}}}
\def\notcon{
  \renewcommand\stacktype{L}\mathrel{\ensurestackMath{%
  \ThisStyle{\stackon[0pt]{\SavedStyle\con}{\SavedStyle\mathbackslash}}}}%
}

\let\separ=\notcon

\newcommand{\prePt}{\mathbf{Q}_G}
\newcommand{\prePtW}{\mathbf{Q}_W}
\newcommand{\Abs}{\mathbf{A}}

\let\Grz=\Pt
\newcommand\Eq{\mathbf{Eq}}
\newcommand\Wthd{\mathbf{W}}

\newcommand{\topo}{\mathscr{O}}

\DeclareMathOperator{\RO}{RO} 

\newcommand{\fil}{\mathscr{F}}

\DeclareMathOperator\downop{\downarrow} 

\newcommand\twodownop{\raisebox{-1pt}{$\mathop{\rotatebox{90}{$\twoheadleftarrow$}}$}} 

\newfont{\eurxi}{eurm10 at 10.95pt}
\newfont{\eurviii}{eurm7 at 8pt}
\newfont{\eurvii}{eurm7}

\newcommand\Real{\mathds{R}} 

\def\dywiz{\kern0sp\discretionary{-}{}{-}\penalty10000\hskip0sp\relax}

\newcommand\Klass{\boldsymbol{\mathsf{K}}} 
\newcommand\BA{\boldsymbol{\mathsf{BA}}} 
\newcommand\BCA{\boldsymbol{\mathsf{BCA}}} 
\newcommand\TBCA{\boldsymbol{\mathsf{TBCA}}} 
\newcommand\BWCA{\boldsymbol{\mathsf{BWCA}}} 

\newcommand\Even{\mathds{E}}
\newcommand\Odd{\mathds{O}}

\newcommand\frp{\mathfrak{p}}
\newcommand\frq{\mathfrak{q}}
\newcommand\frr{\mathfrak{r}}

\DeclareSymbolFont{stx}{OMS}{txsy}{m}{n}
\DeclareMathSymbol{\Ing}{\mathrel}{stx}{118}
\DeclareMathSymbol{\ll}{\mathrel}{stx}{28}
\newcommand{\llt}{\ll_{\mathrm{T}}}

\DeclareMathSymbol{\medcirc}{\mathord}{symbolsC}{7}
\DeclareMathSymbol{\Ext}{\mathrel}{symbolsC}{78}

\newcommand{\covers}{\trianglerighteq} 
\newcommand\covered{\trianglelefteq} 
\newcommand\ncovered{\ntrianglelefteq} 

\DeclareMathSymbol{\nIng}{\mathrel}{symbolsC}{64}
\DeclareMathSymbol{\Ov}{\mathrel}{symbolsC}{7}

\newcommand{\coloneqq}{\mathrel{\mathord{:}\mathord{=}}}

\newsavebox{\putwodownop}
\savebox{\putwodownop}{$\exists\mathord{\twodownop}$}

\begin{document}

\begin{abstract}
One of the main goals of region-based theories of space is to formulate a geometrically appealing definition of \emph{points}. The paper is devoted to the analysis of two such seminal definitions: Alfred N. Whitehead's (\citeyear{Whitehead-PR}) and Andrzej Grzegorczyk's (\citeyear{Grzegorczyk-AGWP}). Relying on the work of Loredana Biacino's and Ginagiacomo Gerla's (\citeyear{Biacino-Gerla-CSGWDP}), we improve their results, solve some open problems concerning the mutual relationship between Whitehead and Grzegorczyk points, and put forward open problems for future investigation.

\medskip


\medskip

\end{abstract}

\maketitle

\section*{Introduction}

Alfred N. Whitehead (\citeyear{Whitehead-PR}) was one of the first thinkers who---following ideas to be found in the seminal paper of \citet{Laguna-PLSSS}---proposed a~geometrically appealing  definition of \emph{point} in terms of \emph{regions of space} and the \emph{contact} relation. His construction was inventive and elegant yet lacked mathematical rigor. In the 1960s, the Polish logician Andrzej Grzegorczyk put forward one of the first mathematically satisfactory systems of region-based topology, in which he formulated a different, yet also geometrically motivated, construction of points. The comparison of the two approaches was carried out by Loredana Biacino and Giangiacomo Gerla (\citeyear{Biacino-Gerla-CSGWDP}), who---under some reasonable assumptions---demonstrated that the two notions of \emph{point} coincide.

The seminal paper by \citet{Biacino-Gerla-CSGWDP} is the foundation for our work.  The two main results of the paper were Theorems 5.1 and 5.3. The former establishes that every Grzegorczyk representative of a point is a Whitehead representative; the latter shows that the reverse inclusion holds for those Whitehead representatives that can be represented by countable families of regions.

To prove the first inclusion Biacino and Gerla work with the second-order theory of Grzegorczyk's (\citeyear{Grzegorczyk-AGWP}). We show that the specific axioms can be eliminated in favour of the standard first-order mereotopological postulates. Moreover, we prove that the second-order monadic statement `every Grzegorczyk representative is a Whitehead representative' is equivalent (in the subclass of Boolean weak contact algebras in the sense of \citet{Duntsch-Winter-WCS} in which every region has a non-tangential part)\marginpar{\tiny !!!} to the first-order statement `there are no atoms'. For the completeness of presentation we show that no part of this equivalence holds in the (general) class of Boolean weak contact algebras.

As for the second inclusion, we identify a gap in the proof of Theorem 5.3, and we show that it cannot be carried out without assuming an additional axiom postulating coherence, a mereotopological counterpart of the connectedness property. We also improve the original result by addressing an open problem from \citep{Biacino-Gerla-CSGWDP}. That is, we show that the countability assumption about Whitehead representatives can be eliminated, if we assume a stronger second-order version of the standard mereotopological interpolation axiom.

Moreover, we prove that in complete structures, purely mereological notions are too weak to guarantee the existence of Whitehead representatives of points. The English logician himself envisaged this, but no general proof of this fact exists in the literature so far.\footnote{We elaborate on this further on p.\,\pageref{page:purely-mereological}.}

We also provide various examples of Whitehead points within algebraic structures. This provides evidence for the claim that Whitehead points are mathematically tractable.

More or less from the beginning of the 21st century, Boolean contact algebras (see e.g.,~\citealp{Stell-BCAANATRCC,Bennett-Duntsch-AAT}) have provided the standard mathematical framework for doing region-based topology. This is a~comfortable situation that allows for the unification and comparison of different approaches to point-free theories of space. For this reason, in this paper, we also use the aforementioned algebras. This approach is different from the original approaches of Whitehead and Grzegorczyk, as the former used a \emph{contact} relation as the only primitive, and the latter worked in mereology (the theory of the \emph{part of} relation) extended with contact. From a technical point of view, these differences are irrelevant. At the same time, the unified well-established environment of Boolean contact algebras allows for a precise and clear presentation of both approaches to region-based theories.

The paper is organized as follows: In Section~\ref{sec:WCA}, we review some preliminaries and  introduce the main objects of study, viz., Boolean (weak) contact algebras. In Section~\ref{sec:G-points}, we present two formal accounts of Grzegorczyk points, i.e.,  Grzegorczyk points defined in terms of equivalences classes and Grzegorczyk points understood as filters. Moreover,  we show that these two definitions are  equivalent in the context of Boolean weak contact algebras. Sections \ref{sec:W-points} addresses a formal account of Whitehead points. We study some of their properties and provide examples of such points within regular open algebras. This section witnesses as well a proof of the insufficiency of purely mereological notions for the existence of Whitehead points.  Section \ref{sec:G-points-are-W-points} studies the minimal constraints that a Boolean Weak Contact Algebras has to satisfy to guarantee that every Grzegorczyk representative of a point is a  Whitehead representative. In particular, in this section we strengthen Theorem 5.1 of \cite{Biacino-Gerla-CSGWDP}. In Section~\ref{sec:W-points-are-G-points-countable} we fill the mentioned gap in the original proof of Theorem 5.3 of   \cite{Biacino-Gerla-CSGWDP} and study the logical status of the second-order condition `every Whitehead representative is a Grzegorczyk representative'. In Section~\ref{sec:W-points-are-G-points-general}, we generalize Theorem 5.3  to  Whitehead points of any size.

\section{Weak-contact and contact algebras}\label{sec:WCA}

As usual,  $\neg$, $\wedge$, $\vee$, $\rarrow$, $\iffslim$, $\forall$ and $\exists$  denote the standard logical constants of negation, conjunction, disjunction, material implication, material equivalence, universal and the existential quantifier. We use `$\nexists$' as an abbreviation for `$\neg\exists$'. Moreover, $\iffdef$ means \emph{equivalent by definition}, and $\defeq$ means \emph{equal by definition}. We use $\omega$ to denote the set of natural numbers understood as von Neumann ordinals. For a fixed space $X$ and $x\subseteq X$, $\complement x\defeq X\setminus x$ is the set-theoretical complement of~$x$ in~$X$. $|X|$ is the~cardinal number of a~set $X$, and $\power(X)$ is its power set.

Moreover, let:
\[
\frB=\langle \boldsymbol{B},\mathord{\cdot},\mathord{+},-,\zero,\one\rangle
\]
be a~Boolean algebra (BA for short) with the operations of, respectively, meet, join, and boolean complement; and with the two distinguished elements: the minimum $\zero$ and the maximum $\one$. Elements of the domain will be called \emph{regions}. The class of all Boolean algebras will be denoted by `$\BA$'. We will often refer to the domain of BA via its name `$\frB$'. Notice that this convention will not lead to any ambiguities.

In $\frB$ we define two standard order relations:
\begin{align}
x\leq y&{}\iffdef x\cdot y=x\,,\tag{$\mathrm{df}\,\mathord{\leq}$}\\
x<y&{}\iffdef x\leq y\wedge x\neq y\,.\tag{$\mathrm{df}\,\mathord{<}$}
\end{align}
In the former case we say that $x$ is \emph{part} of~$y$ or that $x$ is \emph{below}~$y$, in the latter that $x$ is \emph{proper part} of~$y$ or that $x$ is \emph{strictly below}~$y$.

Any Boolean algebra $\frB$ is turned into a~\emph{Boolean contact algebra} (BCA for short) by extending it to a~structure $\langle\boldsymbol{B},\mathord{\cdot},\mathord{+},-,\zero,\one,\con\rangle$ where $\mathord{\con}\subseteq\boldsymbol{B}^2$ is a~\emph{contact} relation which satisfies the following five axioms:

\begin{gather}
\neg(\zero\mathrel{\mathsf{C}} x),\label{C0}\tag{C0}\\
x\leq y\wedge x\neq\zero\rarrow x\mathrel{\mathsf{C}} y,\label{C1}\tag{C1}\\
x\mathrel{\mathsf{C}} y\rarrow y\mathrel{\mathsf{C}} x,\label{C2}\tag{C2}\\
x\leq y\rarrow\forall_{z\in B}(z\mathrel{\mathsf{C}} x\rarrow z\mathrel{\mathsf{C}} y)\,, \label{C3}\tag{C3}\\
x \con (y+ z) \rarrow x \con y\vee x\con z\,.\label{C4}\tag{C4}
\end{gather}
The complement of $\con$ will be denoted by `$\separ$', and in the case $x\separ y$ we say that $x$ is \emph{separated from} $y$.  The class of all Boolean contact algebras will be denoted by `$\BCA$'. If $\mathord{\con}$ satisfies \eqref{C0}--\eqref{C3}, it is called---after \citet{Duntsch-Winter-WCS}---a~\emph{weak contact} relation and the corresponding structure bears the name of a Boolean \emph{weak contact} algebra (BWCA for short). The class of all weak contact algebras will be denoted by `$\BWCA$'.

We introduce the convention according to which given a~class $\Klass$ of structures and some conditions $\varphi_1,\ldots,\varphi_n$ put upon elements of $\Klass$, $\Klass+\varphi_1+\ldots+\varphi_n$ (or $\Klass+\{\varphi_1,\ldots,\varphi_n\}$) is the subclass of $\Klass$ in which every structure satisfies all $\varphi_1,\ldots,\varphi_n$, e.g.,
\[
\BCA=\BWCA+\eqref{C4}\,.
\]

In $\frB\in\BWCA$ we define an auxiliary relation of \emph{non-tangential} inclusion (or \emph{way below}, \emph{well-inside}) relation:
\begin{equation}\label{df:ll}
x\ll y\iffdef x\separ -y\tag{$\mathrm{df}\,\mathord{\ll}$}\,.
\end{equation}
We also define $x\overl y$ to mean that $x\cdot y\neq\zero$, and take $\mathord{\ext}\subseteq\frB\times\frB$ to be the set-theoretical complement of $\overl$. In the former case we say that $x$ \emph{overlaps} $y$, in the latter, that $x$ \emph{is disjoint from} $y$ or $x$ \emph{is incompatible with} $y$. A structure $\langle \boldsymbol{B},\mathord{\cdot},\mathord{+},-,\zero,\one,\mathord{\overl}\rangle$ is a~standard example of a~BCA.\footnote{The overlap relation is actually the smallest contact relation on a~BCA, see \citep{Duntsch-Winter-CBCA}.} The most well-known interpretation of contact is the \emph{topological} one. For a~fixed space $\langle X,\topo\rangle$ we take the underlying algebra to be either the complete algebra $\RO(X)$ of all regular open subsets of $X$, or its subalgebra $B$.  The Boolean operations\footnote{ $\Int$ and $\Cl$ are the standard topological \emph{interior} and \emph{closure} operators.} are:
\begin{align*}
  x\cdot y &{}\defeq x\cap y \\
  x+ y & {}\defeq\Int\Cl(x\cup y)\\
  -x & {}\defeq\Int\complement x
\end{align*}
and the contact relation is given by:
\begin{equation*}
x\cont y\iffdef\Cl x\cap\Cl y\neq\emptyset\,.
\end{equation*}
Moreover, we have:
\[
x\llt y\iffslim \Cl x\subseteq y\,.
\]
The relation $\cont$ satisfies axioms \eqref{C0}--\eqref{C4}, so any topological contact algebra is in the class $\BCA$.

We may use a~similar interpretation on the whole power set algebra of $X$, i.e., $\langle\power(X),\mathord{\cont}\rangle$ is a Boolean contact algebra (provided $X$ is equipped with a~topology, of course). Observe that despite the algebra being atomic, the contact relation does not collapse to the overlap relation. For example, in the case of $\Real$ with the standard topology, the open intervals $(0,1)$ and $(1,2)$ are disjoint, yet they are in contact since their closures share an atom. However, we may look upon $\cont$ as a~form of an overlap relation since in this special case of the power set algebra, we have:
\[
x\cont y\iffslim\Cl x\overl\Cl y\,,
\]
which usually is not true when we take into account regular open algebras. From this, it follows that closed sets are in contact only if they overlap, and thus the contact between atoms reduces to identity, if the underlying topology is $T_1$.

The following facts are standardly proven to hold in $\BWCA$:
\begin{gather}
x\ll y\rarrow x\leq y\,, \label{eq:llcIngr}\\
x\ll y\wedge y\ll x \rarrow x=y\,,\label{antis-ll}\\
x\ll y \wedge y\leq z \rarrow x\ll z\,, \label{eq:ll-ingr->ll}\\
x\leq y \wedge y\ll z \rarrow x\ll z\,,\label{eq:ingr-ll-ll}\\
x\ll y \wedge y\ll z \rarrow x\ll z\,, \label{eq:ll-ll->ll}\\
x\ll y\iffslim -y\ll -x\,.\label{eq:ll-contraposition}
\end{gather}

\begin{definition}
  An \emph{atom} of a Boolean (contact) algebra is a non-zero region $x$ that is minimal with respect to $\leq$ among non-zero regions. A BWCA is \emph{atomic} iff its underlying BA is atomic iff every non-zero region contains an atom. A BWCA is \emph{atomless} iff it does not have any atoms, i.e., satisfies the following condition:
  \begin{equation}\label{eq:no-atoms}
      (\forall x\in\frB\setminus\{\zero\})(\exists y\in \frB\setminus\{\zero\})\,y<x\,.\tag{$\nexists\mathrm{At}$}
  \end{equation}
\end{definition}

\section{Grzegorczyk points}\label{sec:G-points}
A~\emph{Grzegorczyk representative of a point} (for short: \emph{G-representative})\footnote{Both the term and its abbreviation are adopted from \citep{Biacino-Gerla-CSGWDP}.} in $\frB\in\BWCA$ is a~non-empty set $Q$ of regions such that:
\begin{gather}
\zero\notin Q\,,\tag{r0}\label{r0}\\
(\forall u,v\in Q)(u=v\vee u\ll v \vee v\ll u)\, , \tag{r1}\label{r1}\\
(\forall u\in Q)(\exists v\in Q)\; v\ll u\, , \tag{r2}\label{r2}\\
(\forall x,y\in\frB)\bigl((\forall u\in Q)(u\overl x\wedge u\overl y) \rarrow x\mathrel{\mathsf{C}} y \bigl)\, . \tag{r3}\label{r3}
\end{gather}
Let $\prePt$ be the set of all G-representatives of $\mathfrak B$.
The purpose of the definition is to formally grasp the intuition of a point as a~system of diminishing regions determining a~unique location in space. We call it a~\emph{representative}, since if we understand a~point as a perfect representation of a~location in space, then two different sets of regions may represent the same location (see Figure~\ref{fig:the-same-point} for a geometrical intuition on the Cartesian plane). Further, we will identify such G-representatives to be one point. Although the definition has a strong geometrical flavor, G-representatives may be somewhat strange entities in BCAs that have little to do with spatial intuitions. We will look at some indicative examples. But first, let us go through an example of a~G-representative in a~well-known setting: the reals.

\begin{figure}[t!]
\begin{center}\footnotesize
\begin{tikzpicture}
\draw[line width=1pt,dotted] (-0.25,-0.25) rectangle (3.25,3.25);
\draw[line width=1pt,dotted] (0,0) rectangle (3,3);
\draw[line width=1pt,dotted] (2.75,0.25) rectangle (0.25,2.75);
\draw[line width=1pt,dotted] (2.5,0.5) rectangle (0.5,2.5);
\draw[line width=1pt,dotted] (2.25,0.75) rectangle (0.75,2.25);
\draw[line width=1pt,dotted] (2,1) rectangle (1,2);
\draw[line width=1pt,dotted] (1.75,1.25) rectangle (1.25,1.75);

\draw (1.5,1.5) circle (2cm)
circle (1.75cm)
circle (1.5cm)
circle (1.25cm)
circle (1cm)
circle (0.75cm)
circle (0.5cm)
circle (0.25cm);

\node (0) at (-0.75,1.3) {$Q_1$};
\node (0) at (-0.75,-0.2) {$Q_2$};
\end{tikzpicture}
\end{center}
\vspace*{-12pt}
\caption{$Q_1$ and $Q_2$ representing the same location in two-dimensional Euclidean space}\label{fig:the-same-point}
\end{figure}
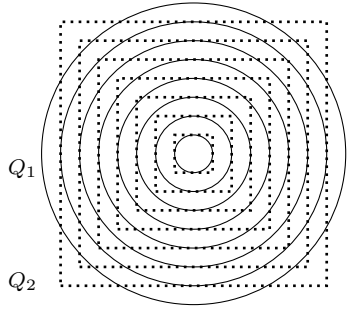

\begin{example}\label{ex:RO(R)}
Take the real line $\Real$ with the Euclidean topology. It is a standard result that the pair $\langle\RO(\Real),\mathord{\cont}\rangle$, where $\RO(\Real)$ is the complete algebra of regular open subsets of $\Real$ and $\mathord{\cont}$ is the standard topological interpretation of contact (as defined above) is a Boolean contact algebra.

Take $0\in\Real$. Obviously, the set:
\[
\textstyle\left\{\left(-\nicefrac{1}{n},\nicefrac{1}{n}\right)\,\middle\vert\,n\in \omega\setminus\{0\}\right\}
\]
is a~G-representative. But also:
\[
\textstyle\left\{\left(-\nicefrac{1}{n},\nicefrac{1}{n}\right)\,\middle\vert\,n\in\Odd\right\}
\]
where $\Odd\subseteq\omega$ is the set of odd numbers, and
\[
\textstyle\left\{\left(-\nicefrac{1}{r},\nicefrac{1}{r}\right)\,\middle\vert\,r\  \text{is a~positive irrational}\right\}
\]
are G-representatives standing for the same location in the one-dimensional space, i.e., number~0. Moreover, one can easily see that there are uncountably many such G-representatives.\footnote{The reader interested in philosophical issues related to Grzegorczyk points is asked to consult~\citep{Gruszczynski-Pietruszczak-SPM}.}
\end{example}

\begin{definition}
If $X,Y$ are subsets of a~BWCA, then $Y$ \emph{covers} $X$ (or $X$ is \emph{covered by} $Y$) iff for every $y\in Y$ there is $x\in X$ such that $x\leq y$. We write `$X\covers Y$' meaning $X$ covers $Y$, and `$X\trianglelefteq Y$' meaning $X$ is covered by $Y$. Let $\ncovered$ be the set-theoretical complement of $\covered$. 

For a~region $x$ of a~BWCA, let $\downop x\defeq\{y\in B\mid y\leq x\}$, i.e., $\downop x$ is the set of all parts of~$x$.
\end{definition}

The general fact that different G-representatives can represent the same location in space follows also from:

\begin{lemma}\emph{\cite[Lemma 5.6]{Gruszczynski-et-al-ASGPFT1}}
If $Q$ is a~G-representative in $\frB\in\BWCA$, then every subset of $Q$ covered by~$Q$ is also a~G-representative. In particular, for any region $x$, $Q\cap\downop x$ is a~G-representative, provided $Q\cap\downop x\neq\emptyset$.
\end{lemma}

\smallskip

In light of the above, to speak about points we need to be able to identify different G-representatives which stand for the same locus.

\subsection{Grzegorczyk points as quotients}\label{sec:G-points-as-quotients}
Let us begin with the definition:
\begin{equation}\tag{$\mathrm{df}\,\mathord{\twodownop}$}
\twodownop x\defeq\{y\in\frB\mid y\ll x\}
\end{equation}
and two lemmas.

\begin{lemma}\label{lem:coi-separ-ext}
If $Q_1$ and $Q_2$ are G-representatives in $\frB\in\BWCA$, then\/\textup{:}
\[
(\forall x\in Q_1)(\forall y\in Q_2)\,x\con y\qquad\text{iff}\qquad Q_2\covered Q_1.
\]
\end{lemma}
\begin{proof}
(i) Assume that $Q_2$ is not covered by $Q_1$, i.e., there is $x_1\in Q_1$ such that for every $y\in Q_2$, $y-x_1\neq\zero$.\footnote{`$x-y$' abbreviates `$x\cdot-y$'.}  By \eqref{r2}, there is $x_0\in Q_1\cap\twodownop x_1$ (i.e., by definition of $\ll$ we have $x_0\separ -x_1$). Observe that for every $z,y\in Q_2$, $z\overl y-x_1$. Indeed, if $z,y\in Q_2$, we have that either (a) $z\leq y$ or (b) $y\leq z$. If (a) holds, $z-x_1\leq z$ and $z-x_1\leq y-x_1$. If (b) holds, $y-x_1\leq z$. If $(\forall z\in Q_2)\,z\overl x_0$, then by \eqref{r3} we obtain that $x_0\con y-x_1$, a~contradiction. So there is $z_0\in Q_2$ such that $z_0\ext x_0$. By \eqref{r2} again there is $z_1\in Q_2\cap\twodownop z_0$, so $z_1\separ x_0$.

\smallskip

(ii) Suppose there are $x\in Q_1$ and $y\in Q_2$ such that $x\notcon y$, but $Q_2\covered Q_1$. Take $z\in Q_2\cap\downop x$. If $z\leq y$, then $y\overl x$, and if $y\leq z$, then $y \leq x$, a~contradiction, as $x\ext y$ by the fact that $\mathord{\overl}\subseteq\mathord{\con}$, which is easily verified.
\end{proof}

In consequence we have:

\begin{corollary}\label{cor:coi-equivalence}
Let $\frB\in\BWCA$. If $Q_1$ and $Q_2$ are G-representatives and $Q_1\covered Q_2$, then $Q_2\covered Q_1$.
\end{corollary}

\begin{theorem}
If $\frB\in\BWCA$, then $\covered$ is an equivalence relation on the set of G-representatives.
\end{theorem}
\begin{proof}
The symmetry of $\covered$ holds by  Corollary \ref{cor:coi-equivalence}. The reflexivity and transitivity of $\covered$ follow from the reflexivity and transitivity of~$\leq$.
\end{proof}

We are now in a position to say precisely that G-representatives $Q_1$ and $Q_2$ \emph{represent the same location} if and only if $Q_1$ is covered by $Q_2$ and $Q_2$ is covered by $Q_1$.
Therefore, it is reasonable to define \emph{points} as equivalence classes of $\covered$ on the set of all G-representatives~$\prePt$ (to emphasize the fact that $Q_1$ and $Q_2$ represent the same location, i.e., are mutually covered by each another, we will write `$Q_1\sim Q_2$'):
\[
\Eq\defeq\prePt/_{\mathord{\sim}}\,.
\]

For any two sets of regions $X$ and~$Y$ such that $X$ covers $Y$ and $Y$ covers $X$, we will say that $X$ and $Y$ are \emph{coinitial}.

\subsection{Grzegorczyk points as filters}\label{sec:G-points-as-filters}
The second (chronologically the first) idea---used by \citet{Grzegorczyk-AGWP}---is to define points as filters that are generated by G-representatives:
\begin{equation}\notag
\fil\ \text{is a point iff}\ (\exists Q\in\prePt)\,\fil=\{x\in\frB\mid(\exists q\in Q)\,q\leq x\}\,.
\end{equation}
By `$\fil_Q$', we will denote a~point generated by the G-representative~$Q$. These filters will be called \emph{G-points}, and the set of all G-points will be denoted by `$\Grz$', while its elements by small fraktur letters `$\frp$', `$\frq$' and `$\frr$', indexed if necessary. For every G-point $\fil_Q$ we have:
\begin{equation}\label{eq:conditions-for-in-FQ}
  x\in\fil_Q\iffslim(\exists y\in Q)\,y\ll x\iffslim(\exists y\in Q)\,y\leq x\,.
\end{equation}

Observe that for $Q_1,Q_2\in\prePt$:
\begin{equation}
Q_1\sim Q_2\iffslim(\exists \frp\in\Grz)\,Q_1\cup Q_2\subseteq\frp\,.
\end{equation}
\begin{proof}
($\rarrow$) This follows from Corollary~\ref{cor:coi-equivalence}, the definition of G-points and~\eqref{eq:conditions-for-in-FQ}.

\smallskip

($\leftarrow$) Let $X=\fil_Q$. Since both $Q_1$ and $Q_2$ are subsets of $\fil_Q$, they must be coinitial with $Q$, and so $Q_1\sim Q_2$.
\end{proof}
Thus, as we see, by considering points as filters we can recover the equivalence relation between G-representatives. The reverse transition---from equivalence classes to filters---is obvious, since for a~given class $[Q]_{\sim}$ it is enough to take $\fil_Q$.

Let us conclude this section with an observation that there is a~1--1 correspondence between G-points as equivalence classes and G-points as filters:
\begin{lemma}
Let $\frB\in\BWCA$. The function $f\colon\Eq\fun\Grz$ such that $f([Q]_{\sim})\defeq\fil_Q$ is a~bijection.
\end{lemma}
\begin{proof}
If $[Q_1]\neq [Q_2]$, then $Q_1\ncovered Q_2$. Therefore, by Lemma~\ref{lem:coi-separ-ext}(i) there are $x\in Q_1$ and $y\in Q_2$ such that $x\ext y$. So $\fil_{Q_1}\neq\fil_{Q_2}$, since otherwise both $x$ and $y$ are in the same filter. Surjectivity is obvious, since every G-point is a filter $\fil_Q$, for some G-representative~$Q$.\end{proof}

\section{Whitehead points}\label{sec:W-points}

In this section, we present a~mathematical analysis of (a representative of) a point as formulated by \citet{Whitehead-PR}, and then used by \citet{Biacino-Gerla-CSGWDP}. Roughly, the idea is that Whitehead points are minimal elements of a poset of abstractive sets of a Boolean weak contact algebra.\footnote{More about the philosophy of and motivations for Whitehead points can be found in two excellent papers by \citet{Gerla-PFC} and \citet{Varzi-PHOC}.}

We begin with the crucial definition:

\begin{definition}\label{df:abstractive-set}
  A set of regions $A$ of a BWCA is an \emph{abstractive set} iff it satisfies \eqref{r0}, \eqref{r1} and:
  \begin{equation}\tag{A}\label{A}
    (\nexists x\in\frB)(\forall y\in A)\,x\leq y\,.
  \end{equation}
The class of all abstractive sets of a~given BWCA is denoted by `$\Abs$'. Since by \eqref{r1} and \eqref{eq:llcIngr} every abstractive set is a~chain w.r.t. $\leq$, it must be the case that for every $x\in A$ there is $y\in A$ such that $y<x$. So, by the Axiom of Dependent Choices, every abstractive set is infinite.
\end{definition}

The idea behind the definition is that we can abstract geometrical objects---like lines, segments, and points---from other entities. However, unlike representatives of Grzegorczyk's, these entities do not have to represent points  but might be planes, straight lines, line segments, triangles, and so on. To use a~simple example, we take the algebra $\RO(\Real^2)$ and the family of regular open sets of the form:
\[
\left\{\langle x,y\rangle\,\middle\vert\, y\in \left(-\nicefrac{1}{n},\nicefrac{1}{n}\right)\right\}\quad\text{for $n\in\omega\setminus\{0\}$}\,,
\]
which is an abstractive set that represents the straight line $y=0$ (i.e., some object from beyond the domain $\RO(\Real^2)$). Of course, we easily see that it is not a~G-representative, since regions:
\[
\{\langle x,y\rangle\mid x\geqslant 1\wedge y\in (-1,1)\}\quad\text{and}\quad \{\langle x,y\rangle\mid x\leqslant -1\wedge y\in (-1,1)\}
\]
overlap all regions from the abstractive set, but are not in~contact (in the sense of $\cont$ for the algebra $\RO(\Real^2)$). So the set violates~\eqref{r3}.

We use the same terminology and symbols for the \emph{covering} relation between abstractive sets that has been introduced for Grzegorczyk representatives. In particular, recall that $A$ and $B$ are coinitial in the case where $A$ is covered by $B$ and $B$ is covered by $A$.



Unlike in the case of covering relation on G-representatives, covering on abstractive sets does not have to be an equivalence relation since it is not---in general---symmetric. However, it is reflexive and transitive, so the coinitiality on abstractive sets is an equivalence relation. Following Whitehead, we will call every element of $\Abs/_{\mathord{\sim}}$ a \emph{geometrical element}. Given $A\in\Abs$ its equivalence class w.r.t. $\sim$ will be denoted by `$[A]$'. If $A_1,A_2\in\Abs$, define a binary relation on $\Abs/_{\mathord{\sim}}$:
\[
[A_1]\preceq[A_2]\iffdef A_1\covered A_2\,.
\]
The relation $\preceq$ is clearly a~partial order. Moreover,  $A_1\sim A_2$ and $B_1\sim B_2$ together entail that: $A_1\covered B_1$ iff $A_2\covered B_2$, thus $\preceq$ is well-defined.
\begin{definition}
For $A\in\Abs$, $[A]$ is a~\emph{Whitehead point} (\emph{W-point}) iff $[A]$ is minimal in $\langle\Abs/_{\mathord{\sim}},\mathord{\preceq}\rangle$. The set of all Whitehead points will be denoted by `$\Wthd$'. $A\in\Abs$ is a \emph{W-representative} of a~point iff $[A]\in\Wthd$. Let $\prePtW$ be the set of all W-representatives of a~given BWCA.
\end{definition}

Alternatively, for an abstractive set $A$ we have:
\begin{equation}\label{eq:alternative-W-representative}
  A\in\prePtW\iffslim\forall_{B\in\Abs}\,(B\covered A\rarrow A\covered B)\,.
\end{equation}

Recall that a BWCA is \emph{atomic} iff its underlying BA is atomic iff every non-zero region contains an atom (i.e., an element that is minimal w.r.t. the standard Boolean order). As an immediate consequence of the definition of an abstractive set we get the following:
\begin{corollary}
    If $\frB\in\BWCA$ is atomic, then $\frB$ does not have any abstractive sets, more so it does not have W-representatives.
\end{corollary}

\subsection{W-representatives in regular open algebras}

If $\RO(X)$ is a~regular open algebra and $A$ is its W-representative, then of course $\bigwedge A=\zero$. However, this does not exclude the possibility in which $\bigcap A\neq\emptyset$, as $\bigwedge A=\Int\bigcap A$. Thus we may ask about set-theoretical intersections of abstractive sets.
\begin{lemma}
If $X$ is a topological space, and $\langle\RO(X),\mathord{\cont}\rangle$ is its topological contact algebra, then for every abstractive set $A\subseteq\RO(X)$, $\bigcap A$ is closed.  Therefore if $\bigcap A\neq\emptyset$ and $\bigcap A\in\RO(X)$, then the space $X$ is disconnected.
\end{lemma}
\begin{proof}
Fix an abstractive set $A$ whose elements are from a~$\langle\RO(X),\mathord{\cont}\rangle$. According to the characterization of $\llt$, $A$ and $\Cl[A]=\{\Cl a\mid a\in A\}$ are coinitial. Indeed, if $a\in A$, then by \eqref{r2} there is $b\in A$ such that $b\llt a$, i.e., $\Cl b\subseteq a$. So $A$ covers $\Cl[A]$. The other direction is obvious since $a\subseteq\Cl a$. Thus,  $\bigcap A=\bigcap\Cl[A]$, and in consequence $\bigcap A$ is closed in~$X$.

Since the infima in $\RO(X)$ are given by the interiors of the intersections, if $X$ is connected, $\bigcap A$ is never an element of the algebra if non-empty.
\end{proof}

This lemma gives rise to a~philosophical interpretation of abstractive sets. If the underlying regular algebra is composed of sets that are models of objects from the physical space (spatial bodies), it usually is a sub-algebra of $\RO(\Real^n)$, where $\Real^n$ is given the standard topology. Various choices are possible\footnote{See for example \citep{DelPiero-CFRUSCM,DelPiero-NCFR}, \citep{Lando-et-al-ACORRBMAT}.}, yet irrespective of these for no abstractive set $A\subseteq\RO(\Real^n)$, $\bigcap A\neq\emptyset$. In this sense abstractive sets represent objects from beyond the universe of models of spatial bodies, i.e., serve as abstraction processes to introduce objects that may be called \emph{geometrical}, \emph{ideal} or, precisely, \emph{abstract}. These objects are, of course, elements of the power set algebra of $\Real^n$, but the idea is that there are <<too many>> objects in $\power(\Real^n)$ from the perspective of the physical space, yet some of the elements of $\power(\Real^n)$ can be treated as approximations made via elements of subalgebras of the regular open algebra of~$\Real^n$.

The definition of a W-representative from the point of view of Euclidean spaces seems to be neat and grasp a certain way in which we may abstract points as higher-order objects. However, in the sequel, we will point to <<strange>> examples. But first, we prove that there are contact algebras that have W-points. Their existence stems from the following:
\begin{theorem}\label{th:local-basis-is-W}
    Let $\langle X,\topo\rangle$ be a topological space. If $A$ is an abstractive set in $\langle\RO(X),\mathord{\cont}\rangle$ that is at the same time a local basis at a~point $p\in X$, then $A$ is a W-representative.
\end{theorem}
    \begin{proof} Suppose $A$ is an abstractive set  in $\langle\RO(X),\mathord{\cont}\rangle$ and a local basis at~$p$. Let $B \subseteq \RO(X)$ be an abstractive set such that $B\covered A$. We show that $p\in\bigcap B$. Suppose otherwise, i.e., let $b_0 \in B$ be such that $p \notin b_0$. Since $B$ is an abstractive set it follows that there exists a $b_1 \in B$ such that $b_1 \llt b_0 $ i.e., $\Cl b_1\subseteq b_0$ and thus $p \notin \Cl  b_1$. Therefore, we have $p \in  X \setminus  \Cl  b_1$, where $X \setminus  \Cl  b_1$ is an open set in $\topo$. It follows that there exists an $a \in A$ such that $a \leq X\setminus  \Cl  b_1 $ and  $p \in a$. Hence,  $a \cdot b_1 = \zero$. In consequence, there exists no $b \in B$ such that $b \leq a$, which contradicts our initial assumption that $A$ covers $B$.

    Since  $p \in\bigcap B$ and $A$ is a~local basis at~$p$, we know that for every $b \in B $ there exists an $a \in A$ such that $a \subseteq b$, so $A$ must be covered by~$B$. Thus,  $A$ is a~W-representative.
\end{proof}
In consequence we have:
\begin{corollary}
    The real line with the standard topology has a W-representative at every point of the space.
\end{corollary}

\begin{definition}[\citealp{Davis-SLOLB}]
    A \emph{lob-space} is a~topological space that at every of its point has a~local basis linearly ordered by the subset relation.
\end{definition}
\begin{definition}[\citealp{Gruszczynski-NTP}]
    A topological space~$X$ is \emph{concentric} iff it is $T_1$ and at every $p\in X$ there is a local basis $\Basis^p$ such that:
    \begin{equation}\tag{R1}\label{eq:R1}
        (\forall U,V\in\Basis^p)\,(U=V\vee \Cl U\subseteq V\vee\Cl V\subseteq U)\,.\qedhere
    \end{equation}
\end{definition}
Thus, concentric spaces are those $T_1$-spaces whose all points have local bases that satisfy the topological version of \eqref{r1} condition for G-representatives. The theorem below demonstrates that these are a subclass of Davies's lob-spaces.
\begin{theorem}[\citealp{Gruszczynski-et-al-GPFBCA}]
    A topological space $X$ is concentric iff it is a~regular lob-space.
\end{theorem}

\begin{theorem}\label{th:concentric-W-representatives}
    If $X$ is a concentric space whose regular open algebra is atomless, then at every point there is a local basis that is a W-representative.
\end{theorem}
\begin{proof}
Since $X$ is regular, it is also semi-regular, so $\RO(X)$ is its basis, which is atomless by assumption. Therefore, the local basis at any point $p$ that satisfies \eqref{eq:R1} must be an abstractive set. So by Theorem~\ref{th:local-basis-is-W} the basis must be a~W-representative.
\end{proof}



Moreover, we have the following result regarding W-representatives, which shows that in the case of topological interpretation, they represent <<small>> chunks of the underlying space.

\begin{lemma}
If $X$ is a regular topological space, and $\langle\RO(X),\mathord{\cont}\rangle$ is its topological contact algebra, then for every W-representative  $A\subseteq\RO(X)$, $\bigcap A$ is a nowhere dense subset of $X$.\footnote{The observation and the proof that $\bigcap A$ is nowhere dense is due to \citet{Hart-APONOIOCCOROS}.}
\end{lemma}
\begin{proof}
Assume  that $x\defeq\bigcap A$ has a~non-empty interior. Therefore, there is a~non-empty regular open set $y$ such that $\Cl y\subseteq\Int x$. This in particular means that $y\llt a$, for all $a\in A$. Since the space is regular and the algebra atomless, we can construct a~sequence such that $y_0\defeq y$ and $y_{n+1}\llt y_n$. In consequence for $Y\defeq\{y_n\mid n\in\omega\}$ we have that $A$ covers $A\cup Y$ but not vice versa, since no element of $Y$ contains an element of $A$. So $A$ is not a W-representative.
\end{proof}

What is common for all W-representatives whose existence follows from the above result is that although they do not have infima as subsets for regular open algebras, they do have a non-empty intersection that is precisely the point of the space that they represent as a basis. This raises the question  whether there is a~regular open algebra that has a W-representative whose set-theoretical intersection is non-empty and that may be interpreted as a~<<new>> point, i.e., something similar to a free ultrafilter being treated as a point of a topological space. The answer is positive, and the example is due to Klaas Pieter Hart.

\begin{example}[\citealp{Hart-APONOIOCCOROS}]\label{ex:Hart}
Consider the ordinal space $X\defeq[0,\omega_1)$, where $\omega_1$ is the first uncountable ordinal. Recall that if $x$ and $y$ are closed and unbounded subsets of $X$, then $x\cap y\neq\emptyset$. Due to this, for any open subsets $x$ and $y$ of $X$, if $\Cl x\subseteq y$, then either $\Cl x$ is compact or $y$ contains an interval $[\alpha+1,\omega_1)$, for some $\alpha<\omega_1$. For if $\Cl x$ is not compact, then it must be unbounded, and since $\Cl x\cap\complement y=\emptyset$,  $\complement y$ is bounded, i.e., there is $\alpha<\omega_1$ such that $\complement y\subseteq[0,\alpha]$. Therefore $[\alpha+1,\omega_1)\subseteq y$. The following set $A\defeq\{[\alpha+1,\omega_1)\mid\alpha<\omega_1\}$ consists of clopen---and the more so regular open---subsets of $X$, and is an abstractive set. If $B$ is also an abstractive set such that $A$ covers $B$, then every element $b\in B$ must be unbounded. The more so the closure of every element of $B$ is unbounded, and since for every $b\in B$ there is $b_0$ in $B$ such that $\Cl b_0\subseteq b$, $b$ must contain an interval $[\alpha+1,\omega)$. Thus $B$ covers $A$, and so $A$ is a W-representative in $\RO(X)$. Of course, $\bigcap A=\emptyset$, and the W-point $[A]$ represents the ordinal $\omega_1$ that is absent from $X$.
\end{example}

This example is quite important from the point of view of the hidden assumptions behind Whitehead points.  \citet[p.\,30]{Bostock-WAROP} writes that:
\begin{quote}
[\ldots] Whitehead's construction [\ldots] does actually have the idea of boundedness built into it: only a bounded nest\footnote{A nest is the counterpart of an abstractive set, it is bounded if it contains only bounded regions (actually it is enough that it contains one such region to be considered bounded). } can satisfy Whitehead's definition of a point-nest. (But I do not suppose that Whitehead recognised this.)
\end{quote}
What the example shows then is that boundedness is not built into the idea of Whitehead points. It only is as far as <<natural>> spaces---like Euclidean spaces---are considered, which follows from further results and properties of Grzegorczyk points proven in \citep{Gruszczynski-NTP,Gruszczynski-et-al-ASGPFT1,Gruszczynski-et-al-ASGPFT2}. In an abstract setting, relevant for this paper, the notion of \emph{boundedness} does not have to be considered, as \citeauthor{Hart-APONOIOCCOROS}'s example shows. Also, this example shows that the \emph{connectedness} of regions that constitute W-representatives is not built into the general idea of points in the sense of Whitehead, as every region in the W-representative from the example is topologically disconnected.\footnote{The issue of connectedness of regions in Whitehead's theory is elaborately discussed in \citep{Bostock-WAROP}.}

Let us make another philosophical remark at this point. It may also be the case that we do not know whether a particular subset of a regular open algebra is a W-representative  due to our current state of knowledge. Consider the following example.

\begin{example} Let $\RO(\Real)$ be the complete algebra of  regular open subsets of $\Real$ and $\langle\RO(\Real),\mathord{\cont}\rangle$ its topological contact algebra. Then, consider the following set.
\[
A\defeq\{(-\nicefrac{1}{p}, \nicefrac{1}{p}) \mid  p\defeq \mbox{ max}\{s,t\} \mbox{ where } s,t \mbox{ are  twin primes} \}\,
\]
The twin prime conjecture, i.e., the claim that there exist infinitely many twin primes, is still an unsolved problem within number theory. So we do not know whether $A$ is finite or infinite. This implies that we also do not know  whether $A$ is an abstractive set and thus  a W-representative.\end{example}

Observe that there are BCAs without any W-representatives and, therefore, without any Whitehead points.

\begin{definition}
  An ordinal number $\alpha$ is \emph{even} iff there is an ordinal number $\beta$ such that $\alpha=2\cdot\beta$ (where $\cdot$ is the standard ordinal multiplication). Otherwise, it is \emph{odd}. Let $\Even_\kappa$ and $\Odd_\kappa$ be, respectively, the set of all even and odd ordinals smaller than~$\kappa$.
\end{definition}

\begin{lemma}\label{lem:no-W-points-atomless-C=O} No complete $\frB \in \BCA$ in which $\mathord{\con}=\mathord{\overl}$ has W-representatives.\footnote{If, additionally, $\frB$ is atomless, then it does not have any G-representative either, which is entailed by Theorem~\ref{th:Q-and-A-is-W} proven further.}
\end{lemma}

\begin{proof}
If $\frB$ is finite, then it cannot have any abstractive sets. The more so it cannot have W-representatives.

So suppose $\frB$ is infinite, and let $\langle x_\alpha\mid\alpha<\kappa\rangle$ be an abstractive set, for some limit cardinal $\kappa$. Since we consider the case in which contact is overlap, we have that:
\[
x_0 > x_1>\ldots>x_n>x_{n+1}>\ldots>x_{\beta}>x_{\beta+1}>\ldots
\]
is an abstractive set. For any $\alpha<\kappa$ define: $y_\alpha\defeq x_\alpha-x_{\alpha+1}$ and consider the antichain $\langle y_\alpha\mid \alpha<\kappa\rangle$. Let $\Odd_\kappa$ and $\Even_\kappa$ be, respectively, all odd and all even ordinals smaller than $\kappa$. Divide the sequence into two sub-sequences:
\[
\langle u_\alpha\mid\alpha\in\Odd_\kappa\rangle\quad\text{and}\quad\langle v_\alpha\mid\alpha\in\Even_\kappa\rangle\,,
\]
take the following suprema:
\[
a_\beta\defeq\bigvee\{u_\alpha\mid\alpha\in\Odd_\kappa\setminus (\Odd_\kappa\cap(2\cdot\beta+1))\}\,,
\]
and gather them into $A\defeq\{a_\beta\mid\beta<\kappa\}$. $A$ is an abstractive set covered by $\langle x_\alpha\mid\alpha<\kappa\rangle$, but does not cover this sequence. Therefore the sequence is not a W-representative. In consequence, no abstractive set is a~Whitehead representative.
\end{proof}

Let us round off this section with the following remarks.\label{page:purely-mereological} Lemma~\ref{lem:no-W-points-atomless-C=O} is a mathematical embodiment of what \cite{Whitehead-CN} discovered himself: the purely mereological notion of \emph{parthood} is too weak to represent his concept of \emph{point} as a~collection of regions. Some arguments for this can be found in Whitehead's book, \citep{Bostock-WAROP} and \citep{Varzi-PHOC}. However, their common weaknesses are that (a) they refer to particular kind of regions that invoke the notions of \emph{dimension} and of \emph{shape} (either explicitly or implicitly) and (b) they do not single out precise assumptions. These are arguments, not proofs in a strict mathematical sense. We present a fully-fledged proof, which is general in the sense that we consider regions as abstract elements of any Boolean algebra. What remains to be eliminated is the assumption of completeness. Thus, we put forward the following open problem:
\begin{problem}
Is there an incomplete BCA in which both $\mathord{\con}=\mathord{\overl}$ and there exists a W-representative?
\end{problem}
Observe that Lemma~\ref{lem:no-W-points-atomless-C=O} does not exclude such algebras, as the property of being a W-representative does not have to be preserved for completions of BAs. That is, if $\frB$ is an incomplete BA that has a W-representative~$A$, and $\cl{\frB}$ is the completion of $\frB$, then the structure of~$A$ is preserved by the canonical embedding $e\colon\frB\to\cl{\frB}$. In consequence, we can repeat the reasoning from Lemma~\ref{lem:no-W-points-atomless-C=O} and show that $e[A]$ is not a W-representative.

\section{G-representatives are W-representatives (under additional assumptions)}\label{sec:G-points-are-W-points}

In this section, we are occupied with two problems: (a) what are the minimal conditions for BWCAs that guarantee that every G-representative is a W-representative, and (b) what is the content of the second order monadic statement about the dependency between the two sets of representatives. Theorem \ref{th:Q-and-A-is-W} below is a stronger version of \cite[Theorem 5.1]{Biacino-Gerla-CSGWDP}. \citeauthor{Biacino-Gerla-CSGWDP}'s proof to establish that every G-representative is a W-representative uses the second-order constraints that postulate the existence of Grzegorczyk points. These are their axioms $\mathrm{G}_4$ and $\mathrm{G}_5$, closely related to the original Grzegorczyk axiom from his paper.\footnote{See \citep{Gruszczynski-et-al-ACOTSOPFT} for a comparison of the original axiomatization of \citeauthor{Grzegorczyk-AGWP}'s system with the system of \citeauthor{Biacino-Gerla-CSGWDP}'s.} We prove that the original result of the Italian mathematicians can be substantially improved, as we only assume the axioms for the weak contact relation plus:
\begin{gather*}
(\forall x\neq\zero)(\exists y\neq\zero)\,y\ll x\,,\tag{C5}\label{C5}
\end{gather*}
and the atomlessness of the underlying Boolean algebra. \eqref{C5}  is known as the \emph{non-tangential part} axiom, and  it is equivalent---by \eqref{df:ll}---to:
\begin{equation*}
(\forall x\neq\one)(\exists y\neq\zero)\,x\notcon y\,,
\end{equation*}
the so-called \emph{disconnection} axiom.
In the class $\BWCA$ both these axioms are equivalent to the \emph{extensionality} axiom:
\begin{equation*}
    (\forall z\in B)\,(z\con x \iffslim z\con y)\rarrow x=y\,.
\end{equation*}

Moreover, we show that in the class $\BWCA+\eqref{C5}$ the second-order monadic statement `every G-representative is a W-representative' is equivalent to the first order condition `there are no atoms'. Additionally, in Theorem~\ref{th:independence-no-atoms-G-are-W} we demonstrate that both the implications fail if we omit the axiom~\eqref{C5}. Thus, via the two theorems, we provide answers to both (a) and (b) above.

The first requirement that G-representatives must meet to be W-representatives is that they are abstractive sets. In general, it does not have to be true: there are contact algebras with G-representatives that are not abstractive sets since the former do not have to satisfy~\eqref{A}. \citeauthor{Biacino-Gerla-CSGWDP} do not have to take this into account since their definition of G-representative contains the requirement that it is a set of regions without the minimal element. This, however, is the assumption that is absent from the definition introduced in the original paper by \citeauthor{Grzegorczyk-AGWP}.

In connection with this we have:
\begin{proposition}[\citealp{Gruszczynski-Pietruszczak-GSPFT}]\label{prop:atom-in-Q}
If $\frB\in\BWCA+\eqref{C5}$ and $\frB$ has an atom~$a$, then $\{a\}\in\prePt$.
\end{proposition}
\begin{proof}
Fix an atom~$a$. By \eqref{C5} there exists a~non-zero $b\in B$ such that $b\ll a$. So $b\leq a$ by \eqref{eq:llcIngr}, and thus $b=a$. From this, we can see that the conditions \eqref{r0}--\eqref{r2} are satisfied. For \eqref{r3}, if $x\overl a$ and $y\overl a$, then $a\leq x$ and $a\leq y$, so $x\overl y$.
\end{proof}
So, if a~BWCA has atoms and satisfies \eqref{C5} there are G-representatives, which are not abstractive sets. Thus, in general, it is not the case that $\prePt\subseteq\Abs$, and the natural thought is to eliminate the existence of atoms, especially due to the fact that G-points generated by atoms are---in a way---not very interesting, similarly as are not principal ultrafilters.

Before we do this, we prove a~propoisition that will help us establish the main results of this section.

\begin{proposition}\label{prop:A-r2s}
Let $\frB\in\BWCA$.
\begin{enumerate}[label=(\arabic*)]
\item[(i)] If $A\in\Abs$, then $A$ satisfies the strong version of~\eqref{r2}\/\textup{:}
  \begin{equation}\tag{r2${}^{\mathrm{s}}$}\label{r2s}
  (\forall x\in A)(\exists y\in A)\,(y\ll x\wedge y\neq x)\,.
  \end{equation}
\item[(ii)] If $X$ is a set of regions such that $X\covered Q$ for some G-representative $Q$, then $X$ satisfies \eqref{r3}.
\item[(iii)] \label{cor:A-coi-Q-is-G-rep} If $A\in\Abs$, $Q\in\prePt$ and $A\covered Q$, then $A\in\prePt$ and $[A]_{\mathord{\sim}}=[Q]_{\mathord{\sim}}$.
\end{enumerate}
\end{proposition}
\begin{proof}
(i) For every $x\in A$ there is a $y\in A$ such that $y<x$. But, by \eqref{r1}, either $x\ll y$ or $y\ll x$. Since the former cannot hold in light of \eqref{eq:llcIngr}, we have the latter.

\smallskip

(ii) Assume that for all $x\in X$, $x\con u$ and $x\con v$. If $q\in Q$, then by covering of $X$ by $Q$~and by \eqref{C3} we have that $q\con u$ and $q\con v$. Therefore $u\con v$, by \eqref{r3} for~$Q$.

\smallskip

(iii)  Follows from the previous items (i), (ii) and Corollary~\ref{cor:coi-equivalence}.
\end{proof}

\begin{proposition}\label{prop:separated-parts} If $\frB\in\BWCA+\eqref{C5}+\eqref{eq:no-atoms}$, then every non-zero region of $B$ has two proper parts that are separated from each other.
\end{proposition}
\begin{proof}
Let $x$ be a non-zero element of $\frB$. Since the algebra has no atoms, there is a non-zero $y$ that is a proper part of $x$. So, by the Boolean axioms, there is another non-zero $z<x$ that is incompatible with $y$. But by \eqref{C5}, $z$ must have a~non-tangential part~$z_0$. So $z_0\notcon y$, and both regions are parts of~$x$.
\end{proof}

\begin{theorem}\label{th:Q-and-A-is-W}
If $\frB\in\BWCA+\eqref{C5}$, then $\frB$ is atomless iff in $\frB$ every G-representative is a W-representative.
\end{theorem}
\begin{proof}
Suppose $\frB\in\BWCA+\eqref{C5}+\eqref{eq:no-atoms}$. Observe that if $Q$ is a G-representative of an algebra $\frB$ from the class, then $Q$ is an abstractive set by Proposition~\ref{prop:separated-parts}. Further, $[Q]$ is a~geometrical element. Suppose $A\in\Abs$ is such that $[A]_{\mathord{\sim}}\preceq [Q]$, i.e., $A\covered Q$. Therefore by Proposition~\ref{prop:A-r2s} we have that $A\in[Q]$, which means that $[A]=[Q]$, as required.

On the other hand, if $\frB \in \BWCA+\eqref{C5}$  and $\frB$ has an atom $a$, then by Proposition~\ref{prop:atom-in-Q}, $\{a\}\in\prePt$. Thus $\prePt\nsubseteq\prePtW$.
\end{proof}

The axiom \eqref{C5} cannot be dropped, even in the class of Boolean contact algebras, that is:
\begin{theorem}\label{th:independence-no-atoms-G-are-W}
    There is a $\frB\in\BCA+\neg\eqref{C5}$ in which there are no atoms, and in which $\prePt\nsubseteq\prePtW$; and there is also an algebra from the same class that has atoms and in which $\prePt\subseteq\prePtW$.
\end{theorem}

To\label{page:d-contact} prove the first part of the theorem we provide a general method for constructing contact algebras. Given a $\frB\in\BA$, let $\boldd$ be its non-zero element that we call \emph{distinguished}. By means of it, we define the following relation:
\[
x\cond y\iffdef x\overl y\vee(x\overl\boldd\wedge y\overl\boldd)\,.
\]
It is routine to verify that $\cond$ is a contact relation, i.e., satisfies axioms \eqref{C0}--\eqref{C4}. Observe that the largest contact relation on $\frB$, i.e., $\frB^+\times \frB^+$, is a special case of $\cond$ in which $\boldd=\one$, or more generally, where $\boldd$ is a~\emph{dense} region in $\frB$ (i.e., such that every non-zero $x$ overlaps it).

We have that:
\begin{align*}
x\lld y&{}\iffslim x\leq y\wedge (x\leq-\boldd \vee \boldd\leq y)\\
&{}\iffslim x\leq y-\boldd \vee x+\boldd\leq y\,.
\end{align*}
from which it follows immediately that:
\begin{equation}\label{eq:d-ll-d}
\boldd\lld\boldd\,.
\end{equation}
We also have that:
\begin{corollary}\label{prop:C5-fails-for-d}
    If $x<\boldd$ and $x\neq\zero$, then $x$ does not have any non-tangential part. In consequence any Boolean contact algebra $\langle\frB,\mathord{\cond}\rangle$ fails to satisfy \eqref{C5}.
\end{corollary}

\begin{lemma}\label{lem:d-G-representative}
$\{\boldd\}$ is a G-representative of $\langle B,\mathord{\cond}\rangle$, and $\uparrow\{\boldd\}$ is its only G-point.
\end{lemma}
\begin{proof}
\eqref{r0} holds by the definition of $\boldd$, \eqref{r1} and \eqref{r2} by \eqref{eq:d-ll-d}, and \eqref{r3} by the definition of $\cond$. In consequence, $\uparrow\{\boldd\}$ is a G-point.

Suppose there is a G-point $\frp\neq{}\uparrow\{\boldd\}$. By \citep[Fact 6.28]{Gruszczynski-NTP} there are $x\in\frp$ and $y\geq\boldd$ such that $x\notcond y$, i.e., $y\lld -x$. Therefore either $y\leq-\boldd$ or $\boldd\leq-x$. The first possibility is exlcuded by the fact that $\boldd\neq\zero$ and $\boldd$ is below $y$. Therefore the second one holds, and thus $x\leq-\boldd$.
\end{proof}

\begin{proof}[Proof of the first part of Theorem~\ref{th:independence-no-atoms-G-are-W}]
Take any atomless Boolean algebra $\frB$, fix its distinguished element $\boldd$, and expand it to the Boolean contact algebra $\langle\frB,\cond\rangle$. By Lemma~\ref{lem:d-G-representative}, the singleton $\{\boldd\}$ is a G-representative that is finite and therefore cannot be a~W-representative. By Proposition~\ref{prop:C5-fails-for-d}, the algebra has regions without non-tangential parts, so it fails to satisfy~\eqref{C5}.
\end{proof}

\begin{proof}[Proof of the second part of Theorem~\ref{th:independence-no-atoms-G-are-W}] Consider the following two contact algebras, $\langle\RO(\Real),\mathord{\cont}\rangle$ and the four element Boolean algebra $\frB_4\defeq\{\zero,a,b,\one\}$ with the full contact relation $\con_\one$ (i.e., $\boldd\defeq\one$). Consider their product\label{page:P-algebra} $\frP\defeq \RO(\Real)\times\frB_4$ as Boolean algebras (i.e., all the operations are defined coordinate-wise) but with the contact defined as:
\[
\langle x,u\rangle\con\langle y,w\rangle\iffdef x\cont y\vee u\con_1 w\,.
\]
It is routine to verify that $\con$ satisfies \eqref{C0}--\eqref{C4}.\footnote{This is not, then, the product of the two algebras as the \emph{contact} algebras. The relation $\langle x,u\rangle\mathrel{R}\langle y,w\rangle\iffdef x\cont y\wedge u\con_1 w$ is not contact, as the reader may easily convince themself.} We have that:
\begin{align*}
\langle x,u\rangle\ll\langle y,w\rangle&{}\iffslim x\llt y\wedge u\ll_\one w\\
&{}\iffslim x\llt y\wedge (u=\zero\vee w=\one)\,.
\end{align*}
The algebra has two atoms: $\langle\zero,a\rangle$ and $\langle\zero,b\rangle$. $\frP$ does not satisfy \eqref{C5}, as none of the two atoms is its own non-tangential part. In consequence, neither the singleton of the former nor the singleton of the latter is a G-representative.

Observe that the set of G-representatives of $\frP$ contains all sets of the form $Q\times\{\zero\}$, where $Q$ is a~G-representative of $\langle\RO(\Real),\mathord{\cont}\rangle$. It is quite obvious that \eqref{r0}--\eqref{r2} are satisfied by $Q\times\{\zero\}$. As for \eqref{r3}, if we have pairs $\langle x,u\rangle$ and $\langle y,w\rangle$ that overlap every element of $Q\times\{\zero\}$, then for all $z\in Q$, $x\overl z$ and $y\overl z$, so $x\cont y$, which is enough to conclude that $\langle x,u\rangle\con\langle y,w\rangle$.

The only products that are G-representatives in the algebra are sets of the form $Q\times\{\zero\}$, where $Q$ is a G-representative in $\RO(\Real)$. Firstly, neither $Q\times\{a\}$ nor $Q\times\{b\}$ can be G-representatives, as no element of any of the two sets has a non-tangential part. Secondly, any set of the form $Q\times\{\one\}$, where $Q$ is a G-representative in $\RO(\Real)$, fails to satisfy \eqref{r3}. To see this, take any regular open set $x$ that overlaps every element of $Q$ and consider pairs $\langle x,0\rangle$ and $\langle 0,\one\rangle$. We see that for any $\langle y,\one\rangle\in Q\times\{\one\}$, $\langle x,0\rangle\overl \langle y,\one\rangle$ and $\langle 0,\one\rangle\overl \langle y,\one\rangle$, yet $\langle x,0\rangle\notcon\langle 0,\one\rangle$. Thirdly, any product $Q\times M$, where $M$ is an at least two element subset of $\frB_4$, fails to be a chain and so cannot be a G-representative. Fourthly, if we have a set $M\times\{\zero\}$ where $M$ is not a G-representative, i.e., it fails to meet one of the conditions \eqref{r0}--\eqref{r3}, then since $\zero\ll_\one\zero$, $M\times\{\zero\}$ also fails to meet one of the four conditions for~$\ll$.

Of course, every $Q\times\{\zero\}$ is an abstractive set, and in the case where it covers an abstractive set $A\times\{0\}$, $Q$ must cover $A$ in $\RO(\Real)$. So  by Proposition~\ref{prop:A-r2s} (iii) $A$ is a W-representative in $\RO(\Real)$, and so $A\times\{\zero\}$ is a W-representative in $\frP$.
\end{proof}

\section{W-representatives with countable coinitiality are G-representatives }\label{sec:W-points-are-G-points-countable}

In this section, we reconstruct \citeauthor{Biacino-Gerla-CSGWDP}'s proof according to which every W-representative that can be represented by an $\omega$-sequence (in the sense explained below) is also a G-representative. We aim to show that with respect to the original proof from \citep{Biacino-Gerla-CSGWDP} we need to assume the so-called \emph{coherence} axiom to assure that the machinery works properly.

To be more precise, \citeauthor{Biacino-Gerla-CSGWDP} in Theorem 5.3 prove that if we extend the standard axiomatization for contact with the \emph{interpolation} axiom\footnote{They call it the \emph{normality} axiom.}:
\begin{equation}\tag{IA}\label{IA}
  x\ll y\rarrow(\exists z\in\frB)\,x\ll z\ll y\
\end{equation}
we can prove that every Whitehead representative that can be represented as an $\omega$-sequence is a Grzegorczyk representative. However, to show that a certain sequence of regions is an abstractive set they make a~transition that cannot be justified without an application of the so-called \emph{coherence} axiom, that we introduce below. Thus, in the premises of Theorem~\ref{th:BG-improved} we explicitly assume coherence in the form of \eqref{C6} below. Coherence is a~mereotopological counterpart of topological connectedness.\footnote{See, e.g., \citep{Bennett-Duntsch-AAT} for details.}

\begin{definition}
For a~given chain $C$ let the \emph{coinitiality} of $C$ be the smallest cardinal number~$\kappa$ such that there exists an antitone function $f\colon\kappa\fun C$ with $f[\kappa]$ coinitial with~$C$.
\end{definition}

\begin{definition}
  For a~given $\frB\in\BWCA$, let $\prePtW^\omega$ be the set of all Whitehead representatives whose coinitiality is~$\omega$.
\end{definition}

Observe that the set~$A$ from Example~\ref{ex:Hart} is an instance of a~W-representative whose coinitiality is $\omega_1$. On the other hand, a~local basis of any point $r\in\Real$ (with the standard topology) that satisfies \eqref{eq:R1} is a~W-representative that is an element of $\prePtW^\omega$ in $\RO(\Real)$.

\begin{definition}\label{df:coherent}
A Boolean weak contact algebra is \emph{coherent}\footnote{The term is taken from \citep{Roeper-RBT}.} iff its unity is coherent, iff it satisfies the following \emph{coherence} axiom:
\begin{equation}\tag{C6}\label{C6}
  x\notin\{\zero,\one\}\rarrow x\con -x\,.\qedhere
\end{equation}
\end{definition}

\begin{proposition}\label{prop:C6-ll-in-part} In the class $\BWCA$, \eqref{C6} is equivalent to\/\textup{:}
  \begin{equation*}
  x\notin\{\zero,\one\}\wedge x\ll y\rarrow x<y\,.
  \end{equation*}
\end{proposition}
\begin{proof}
($\rarrow$) If $x$ is neither $\zero$ nor $\one$, then $x\con-x$. Assume that $x\ll y$, i.e., $x\notcon-y$. If $y=\one$, then $x<y$. So suppose $y\neq\one$. Since $x\leq y$, also $y\neq\zero$ and $y\con-y$. Therefore $x\neq y$, which means that $x<y$, as required.

\smallskip

($\larrow$) If $x\notin\{\zero,\one\}$ and $x\notcon-x$, then $x\ll x$, and $x<x$ by the assumption. A~contradiction.
\end{proof}

Recall that the following condition holds in every BCA:
\begin{equation}\label{eq:ll-sum-prod}
x\ll u\wedge y\ll v\rarrow x\cdot y\ll u\cdot v\,.
\end{equation}

\begin{proposition}\label{prop:for-key-theorem}
    In every $\frB\in\BCA$, \eqref{C6} is equivalent to the following condition\/\textup{:}
    \[
    x\neq u\wedge x\ll u\wedge y\ll v\rarrow x\cdot y\ll u\cdot v\wedge x\cdot y\neq u\cdot v\,.
    \]
\end{proposition}
\begin{proof}
($\rarrow$) Assume \eqref{C6}. If $x\neq u\wedge x\ll u\wedge y\ll v$, then by \eqref{eq:ll-sum-prod} we have that $x\cdot y\ll u\cdot v$, so by \eqref{C6} we obtain: $x\cdot y\neq u\cdot v$,

\smallskip

($\larrow$) Suppose there is $x\notin\{\zero,\one\}$ such that $x\ll x$. Since $x\ll\one$ and $x\neq\one$, by the hypothesis we have that $x\cdot x\neq x\cdot\one$, a contradiction.
\end{proof}

\begin{theorem}[after \citealp{Biacino-Gerla-CSGWDP}]\label{th:BG-improved}
  If $\frB\in\BCA+\eqref{IA}+\eqref{C6}$, then $\prePtW^\omega\subseteq\prePt$.
\end{theorem}
\begin{proof}
Let $A$ be an~abstractive set, and let $(x_i)_{i\in\omega}$ be its coinitial subset. Assume it is not  a~G-representative, i.e., that it fails to satisfy~\eqref{r3}. Let then~$u$ and~$v$ be such that for every $i\in\omega$, $u\overl x_i\overl v$, but $u\ll-v$. Observe that $u\neq\zero\neq v$. By~\eqref{IA} (and by the Axiom of Dependent Choices), there is a~sequence $(u_i)_{i\in\omega}$ such that:
\[
u\ll\ldots u_2\ll u_1\ll u_0=-v\,.
\]
We have that for every $i,j\in\omega$, (a) $u_i\cdot x_i\neq\zero$ and (b) $x_j\nleq u_i\cdot x_i$ (since $u_i\cdot x_i\ll -v$ and $x_j\overl v$, i.e. $x_j\nleq-v$). Observe that $(u_i\cdot x_i)_{i\in\omega}$ is an abstractive set. \eqref{r0} is a~consequence of~(a), \eqref{r1} holds for the sequence because of \eqref{eq:ll-sum-prod}: $u_{i+1}\cdot x_{i+1}\ll u_i\cdot x_i$. But $u_i\cdot x_i\neq\one$, so thanks to Proposition~\ref{prop:C6-ll-in-part} we have that $(u_i\cdot x_i)_{i\in\omega}$ satisfies~\eqref{A}. The sequence is covered by  $(x_i)_{i\in\omega}$, but $(x_i)_{i\in\omega}$ is not covered by $(u_i\cdot x_i)_{i\in\omega}$ by (b), so the former cannot be a~W-representative. So $A$ is not a~W-representative, either.
\end{proof}

As it can be seen in the proof above, there are two ways two justify the conclusion that $(u_i\cdot x_i)_{i\in\omega}$ meets the condition~\eqref{A}: either by an application of Proposition~\ref{prop:C6-ll-in-part} or by a~reference to Proposition~\ref{prop:for-key-theorem}. Both these conditions are equivalent to \eqref{C6}, so there is no way to escape the coherence axiom in the construction. This, of course, does not show that `$\prePtW^\omega\subseteq\prePt$' is independent from $\BCA+\eqref{IA}$, as it only means that \emph{the proof} itself requires \eqref{C6}. Nor does it undermine the construction from the proof, as the remedy is relatively simple and only calls for the explicit assumption of the axiom. However, it would be desirable to fully know the status of the coherence axiom with respect to the sentence `$\prePtW^\omega\subseteq\prePt$'. As we have not been  able to settle it, we put forward the following:
\begin{problem}
    Show that `$\prePtW^\omega\subseteq\prePt$' cannot be deduced from the axioms for Boolean contact algebras extended with the interpolation axiom. That is, find an algebra $\frB\in\BCA+\eqref{IA}+\eqref{eq:no-atoms}$ that has a~W-representative which is not a~G-representative. Similarly, show that `$\prePtW^\omega\subseteq\prePt$' is not true in some $\frB\in\BCA+\eqref{C6}+\eqref{eq:no-atoms}$. We add the assumption about the non-existence of atoms, since Biaciono and Gerla have it among their postulates.
\end{problem}

It was proven in \citep{Duntsch-et-al-ARAATTRCC} that those Boolean contact algebras that satisfy \eqref{C5} and \eqref{C6} must be atomless.  In light of this and Theorems~\ref{th:Q-and-A-is-W} and \ref{th:BG-improved} we have:
\begin{theorem}\label{th:conclusion}
    If $\frB\in\BCA+\eqref{IA}+\eqref{C5}+\eqref{C6}$, then $\prePtW^\omega\subseteq\prePt\subseteq\prePtW$. If additionally $\frB$ satisfies the countable chain condition, then $\prePtW=\prePt$.
\end{theorem}

Recall that a~topological space $X$ is \emph{semi-regular} iff it has a~basis that consist of regularly open subsets of~$X$. It is \emph{weakly regular} iff for every non-empty open set $M$ there exists a non-empty open set $K$ such that $\Cl K\subseteq M$. $X$ is \emph{$\kappa$-normal}\footnote{These spaces were introduced and studied by \citet{Shchepin-RFNNS}.} (or \emph{weakly normal}) iff any pair of disjoint regular closed sets can be separated by disjoint open sets.

By \citep[Proposition 3.7]{Duntsch-et-al-RTBCA} we have that for a space $X$ and a dense subalgebra $\frB$ of $\langle\RO(X),\mathord{\cont}\rangle$ the following correspondences hold:
\begin{enumerate}[label=(\roman*)]
    \item $\cont$ satisfies \eqref{C5} iff $X$ is weakly regular,
    \item $\cont$ satisfies \eqref{C6} iff $X$ is connected,
    \item $\cont$ satisfies \eqref{IA} iff $X$ is $\kappa$-normal.
\end{enumerate}
Let $\TBCA$ be the class of all \emph{topological} contact algebras, that is those of the form $\langle \frB,\mathord{\cont}\rangle$, where $\frB$ is a subalgebra of a~regular open algebra $\RO(X)$ of a~topological space~$X$.

\begin{lemma}[{\citealp[Lemma 3.56]{Bennett-Duntsch-AAT}}]
    $\frB\in\TBCA+\eqref{IA}+\eqref{C5}+\eqref{C6}$ iff $\frB$ is a~dense subalgebra of $\langle\RO(X),\mathord{\cont}\rangle$, where $X$ is a~$\kappa$-normal, connected $T_1$-space.
\end{lemma}

In consequence, by this and by Theorem~\ref{th:conclusion} we obtain:
\begin{corollary}
    If $X$ is a~$\kappa$-normal, connected $T_1$-space, and $\frB$ is a~dense subalgebra of $\langle\RO(X),\mathord{\cont}\rangle$, then in $\frB$\/\textup{:} $\prePtW^\omega\subseteq\prePt\subseteq\prePtW$. If additionally $X$ as a~topological space satisfies the countable chain condition, then $\prePtW=\prePt$.
\end{corollary}

Theorem~\ref{th:Q-and-A-is-W} shows that the second-order statement `$\prePt\subseteq\prePtW$' corresponds to a~first-order property of atomlessness of Boolean algebras. It is then natural to ask if the statements `$\prePtW^\omega\subseteq\prePt$' and `$\prePtW\subseteq\prePt$' correspond to any familiar, not necessarily first-order, properties of BCAs, or topological BCAs. Let us focus on this, and let us prove some negative results.

Observe that the coherence axiom cannot be deduced by means of `$\prePtW\subseteq\prePt$' (more so, by means of `$\prePtW^\omega\subseteq\prePt$'), in the following sense:
\begin{proposition}\label{prop:C6-cannot-be-deduced}
    \eqref{C6} is not true in $\BCA+\eqref{IA}+\eqref{C5}+\eqref{eq:no-atoms}+\prePtW\subseteq\prePt$.
\end{proposition}
\begin{proof} Take the contact algebra $\langle\RO(\Real),\mathord{\overl}\rangle$ and apply Lemma~\ref{lem:no-W-points-atomless-C=O}.
\end{proof}

If the reader finds the reference to Lemma~\ref{lem:no-W-points-atomless-C=O} somewhat sneaky (as there are no W-representatives in BCAs that satisfy the premises), we can construct an example of a~BCA that has W-representatives and meets the conditions of the proposition but not \eqref{C6} by taking the topological space $X\defeq[0,1]\cup[2,3]$ as the subspace of the reals with the standard topology and considering $\langle\RO(X),\cont\rangle$. $X$ is normal (more so, $\kappa$-normal), $T_1$, satisfies the countable chain condition. Thus $\prePtW\subseteq\prePt$. $\RO(X)$ is obviously atomless.  Moreover, $X$ is metrizable, so it's a~concentric space and thus has W-representatives by Theorem~\ref{th:concentric-W-representatives}. But $X$ has two non-trivial components, and thus $\cont$ does not satisfy \eqref{C6}. Does adding the assumption $\prePtW\neq\emptyset$ does not improve the situation.

Making a suitable modification to $X$, e.g., taking $Y\defeq[0,1]\cup\{2\}$ we see that:
\begin{proposition}
    \eqref{eq:no-atoms} is not a~consequence of $\BCA+\eqref{IA}+\eqref{C5}+\prePtW\subseteq\prePt$.
\end{proposition}
\noindent Of course, by Theorem~\ref{th:Q-and-A-is-W}, in $\RO(Y)$ we have $\prePt\nsubseteq\prePtW$, specifically with $\{\{2\}\}$ being the culprit.

As for the relation of \eqref{C5} to `$\prePtW\subseteq\prePt$', observe that since the axiom is not a~consequence of $\{\text{\eqref{C0}--\eqref{C4}},\eqref{IA},\eqref{C6},\eqref{eq:no-atoms}\}$, it cannot be a~consequence of $\{\text{\eqref{C0}--\eqref{C4}},\text{`$\prePtW^\omega\subseteq\prePt$'}\}$ by Theorem~\ref{th:BG-improved}. Thus there may be no deeper connection between the two statements.

The above are selective remarks that show what cannot be proven by or about the inclusion `$\prePtW\subseteq\prePt$'. So, we put forward other open problems:
\begin{problem}
    Find $\frB\in\BCA+\eqref{C6}+\prePtW\subseteq\prePt$ in which the interpolation axiom fails.
\end{problem}
\begin{problem}
    Characterize the class $\BCA+\prePtW\subseteq\prePt$, with possibly additional constraints.
\end{problem}
We also ask:
\begin{problem}
    Is there any topological property of a space $X$ that corresponds to the statement $\prePtW\subseteq\prePt$, as a~second order statement formulated about a~subalgebra of $\langle\RO(X),\mathord{\cont}\rangle$?
\end{problem}

\section{W-representatives are G-representatives (the general case)}\label{sec:W-points-are-G-points-general}

Having proven Theorem 5.3 (of whose counterpart is---by a fortunate coincidence---our Theorem~\ref{th:BG-improved}) \citeauthor{Biacino-Gerla-CSGWDP} write:
\begin{quotation}
  [\ldots] the just proven theorem holds for W-representatives that are expressible by [countable] sequences. We do not know if this result holds in any case.
\end{quotation}
In this section, we prove that the we can extend Theorem~\ref{th:BG-improved} to W-representatives of an arbitrary coinitiality provided we assume the \emph{generalized} version of the interpolation axiom.

Observe that by \eqref{eq:ll-sum-prod} the interpolation axiom is equivalent to the following:
\[
x\ll x_1\wedge\ldots\wedge x\ll x_n\rarrow (\exists z\in\frB)\,(x
\ll z\wedge z\ll x_1\cdot\ldots\cdot x_n)\,.
\]
Thus its natural extension to infinite cases is the following second-order constraint:
\begin{equation}\tag{GIA}\label{GIA}
    (\forall Y\in\power(\frB))\,((\forall y\in Y)\,x\ll y\rarrow(\exists z\in\frB)\,(x\ll z\wedge(\forall y\in Y)\,z\ll y))\,.
\end{equation}

Observe that \eqref{GIA} axiom puts a serious constraint upon the existence of W-representatives in the class of complete contact algebras:
\begin{theorem}\label{th:GIA+C5-no-W}
If $\frB\in\BCA+\eqref{C5}+\eqref{GIA}$ is complete, then there are no Whitehead representatives in $\frB$.
\end{theorem}

\begin{proof}
It is easy to show that in complete Boolean contact algebras, \eqref{GIA} is equivalent to the following second-order constraint:
\[
(\forall y\in J)\,x\ll y\rarrow x\ll \bigwedge J
\]
which is equivalent to the following generalized version of \eqref{C4}:
\[
x\con\bigvee J\rarrow(\exists y\in J)\,x\con y\,.
\]
As it has been shown in \citep{Gruszczynski-Menchon-FCRTMOAB}, in presence of \eqref{C5} the latter axiom entails that $\mathrm{\con}=\mathrm{\overl}$. Thus by Lemma~\ref{lem:no-W-points-atomless-C=O}, $\frB$ does not have any W-representatives.
\end{proof}
It follows that if we want to have BCAs that satisfy \eqref{GIA} and have Whitehead representatives, we must drop either \eqref{C5} or completeness, or both. Fortunately, we can demonstrate that such algebras exist.
\begin{theorem}\label{th:GIA-consistency}
Any $\langle\frB,\mathord{\cond}\rangle$ satisfies    \eqref{GIA}.
\end{theorem}
\begin{proof}
Take as the distinguished element of a $\frB\in\BA$ any $\boldd\neq\zero$ and consider $\cond$. Suppose that $A\subseteq\frB$, and let $x$ be such that for all $a\in A$, $x\lld a$, which means that either $x\leq a-\boldd$ or $x+\boldd\leq a$. In the former case, $x\leq-\boldd$, so $x\lld x$, and we are done. In the latter, $x+\boldd\lld a$, as $-a$ must be disjoint from $x+\boldd$. Since it is always true that $x\lld x+\boldd$, we have that $x+\boldd$ is the region that is strongly between $x$ and $a$.
\end{proof}
In light of the above theorem, any Boolean algebra can be turned into a Boolean contact algebra that meets the generalized interpolation axiom. In particular, there will be such algebras that are either complete or incomplete, with or without any atoms, and of arbitrary cardinality. However, in light of Corollary~\ref{prop:C5-fails-for-d}, none of these algebras will satisfy \eqref{C5}.

Having shown the consistency of \eqref{GIA} with the standard axioms for contact, we go on to prove the following:
\begin{theorem}\label{th:BG-general}
  If $\frB\in\BCA+\eqref{GIA}+\eqref{C6}$, then $\prePtW\subseteq\prePt$.
\end{theorem}
\begin{proof}
Suppose $|B|=\kappa$. Fix an abstractive set $A$. Since it is linearly ordered by $\leq$, it must have a~coinitial sequence $\langle x_\alpha\mid\alpha<\lambda\rangle$ for a~limit ordinal $\lambda\leqslant\kappa$. Again, suppose the sequence is not a~G-representative, i.e. it fails to satisfy \eqref{r3}. Let $u$ and $v$ be regions such that each of them overlaps every $x_\alpha$ from the sequence, yet they are separated, i.e., $u\ll-v$. We construct another $\lambda$-sequence repeating the technique from the proof of Theorem~\ref{th:BG-improved}, but applying \eqref{GIA}.

If $\lambda=\omega$, then it is enough to observe that \eqref{IA} is just a special case of \eqref{GIA} where $Y\defeq\{-v\}$. So assume that $\lambda>\omega$.
Suppose $\alpha<\lambda$ is a~limit ordinal and for every $\delta<\beta<\alpha$ we defined $u_\beta$ and $u_\delta$ such that:
\[
u\ll u_\beta\ll u_\delta\ll-v\,.
\]
Consequently, we have that $u$ is a non-tangential part of every element of the sequence $\langle u_\beta\mid\beta<\alpha\rangle$. Thus, by \eqref{GIA} we may choose $u_\alpha$ to be a~region $z$ such that $(\forall \beta<\alpha)\,z\ll u_\beta$ and $u\ll z$. Following this procedure we can construct the $\lambda$-sequence $\langle u_\alpha\mid\alpha < \lambda\rangle$. We go on to show that $\langle x_{_\alpha} \cdot u_{\alpha} \mid \alpha < \lambda \rangle $  is an abstractive set.  \eqref{r0} holds given that we have $u_\alpha  \cdot x_\alpha \neq \zero$ for any  $\alpha < \lambda$.  \eqref{r1} holds due to \eqref{eq:ll-sum-prod}: $u_{\delta}\cdot x_{\delta}\ll u_{\beta}\cdot x_{\beta}$ for any $\beta,  \delta \in \lambda$ such that $\beta < \delta$. Additionally,we have that $u_{\alpha+1}\cdot x_{\alpha +1 }\ll u_{\alpha}\cdot x_{\alpha}$ and $u_{\alpha}\cdot x_{\alpha} \neq\one$  for any  $\alpha < \lambda$. Therefore, by  Proposition~\ref{prop:C6-ll-in-part} the sequence $\langle x_{_\alpha} \cdot u_{\alpha} \mid \alpha < \lambda \rangle $  satisfies~\eqref{A}.

Notice that the sequence $\langle x_{_\alpha} \cdot u_{\alpha} \mid \alpha < \lambda \rangle $  is covered by  $\langle x_\alpha\mid\alpha<\lambda\rangle$. However, we also have that $x_\beta\nleq u_\delta \cdot x_\delta $ for any $\beta,  \delta \in \lambda$ since $u_\delta \cdot x_\delta \ll -v$ and  $x_\beta \nleq-v$. Therefore, $\langle x_\alpha\mid\alpha<\lambda\rangle$
 is not covered by $\langle x_{_\alpha} \cdot u_{\alpha} \mid \alpha < \lambda \rangle $. Thus, $\langle x_\alpha\mid\alpha<\lambda\rangle$ is not a~W-representative and we can conclude that also  $A$ is not a~W-representative.
\end{proof}

\begin{problem}
    In Theorem~\ref{th:GIA-consistency} we have shown that the generalized interpolation axiom is consistent with the standard axioms for BCAs. However, what we do not know is whether there are contact algebras in which both the axioms \eqref{GIA}, \eqref{C6} hold, and there is at least one W-representative. Thus we ask: \emph{are there such BCAs?}
\end{problem}


\section*{Acknowledgements}

This research was funded by the National Science Center (Poland), grant number 2020/39/B/HS1/00216, ``Logico-philosophical foundations of geometry and topology''.

For the purpose of Open Access, the authors have applied a CC-BY public copyright license to any Author Accepted Manuscript (AAM) version arising from this submission.

\appendix

\bibliographystyle{apalike}
\providecommand{\noop}[1]{}

\end{document}